\documentclass[12pt]{amsart}

\usepackage{amssymb}
\usepackage{enumerate}
\usepackage[dvips]{graphicx}
\DeclareGraphicsExtensions{.eps,.art,.ART,.ps}

\usepackage{amsmath, amsthm, amsfonts, color}

\newcommand{\ds}{\displaystyle}

\newtheorem{theorem}{Theorem}
\newtheorem{corollary}[theorem]{Corollary}
\newtheorem{thm}{Theorem}[section]

\newtheorem{lemma}[thm]{Lemma}

\theoremstyle{definition}

\theoremstyle{remark}

\def\g{\gamma}

\def\l{\lambda}
\def\p{\partial}

\def\E{\mbox{\rm e}}
\def\a{\alpha}

\def\Odr{\mathcal{O}}

\def\di{\,\mathrm{d}}

\def\iu{\mathrm{i}}

\def\la{\langle}
\def\ra{\rangle}

\DeclareMathOperator{\RE}{Re}
\DeclareMathOperator{\IM}{Im}

\numberwithin{equation}{section}

\def\NB{\marginpar{\textcolor{red}{$\blacklozenge$}}}

\title[Eigenvalue asymptotics for the damped wave equation]
{Eigenvalue asymptotics, inverse problems and a trace formula for
the linear damped wave equation}
\author{Denis Borisov \and Pedro Freitas}
\address{
Department of Physics and Mathematics, Bashkir State Pedagogical
University, October rev. st., 3a, 450000, Ufa, Russia
}\email{borisovdi@yandex.ru}
\address{Department of Mathematics, Faculdade de Motricidade Humana (TU Lisbon) {\rm and}
Group of Mathematical Physics of the University of Lisbon\\
Complexo Interdisciplinar, Av.~Prof.~Gama Pinto~2\\ P-1649-003
Lisboa, Portugal}\email{freitas@cii.fc.ul.pt}
\date{\today}

\thanks{
D.B. was partially supported by RFBR (07-01-00037) and
gratefully acknowledges the support from Deligne 2004 Balzan
prize in mathematics. D.B. is also supported by the grant of the
President of Russia for young scientist and their supervisors
(MK-964.2008.1) and by the grant of the President of Russia for
leading scientific schools (NSh-2215.2008.1) P.F. was partially supported by FCT/POCTI/FEDER.
 }

\begin{document}

\allowdisplaybreaks

\begin{abstract}
We determine the general form of the asymptotics for Dirichlet
eigenvalues of the one--dimensional linear damped wave operator. As a
consequence, we obtain that given a spectrum corresponding to
a constant damping term this determines the damping term
in a unique fashion. We also derive a trace formula for this
problem.
\end{abstract}
%
%

\maketitle

\section{Introduction}

Consider the one--dimensional linear damped wave equation on the
interval $(0,1)$, that is,

\begin{equation}\label{waveeq}
\left\{
\begin{aligned}
&w_{tt}+2a(x)w_{t}=w_{xx}+b(x)w, && x\in(0,1),\; t>0
\\
&w(0,t) = w(1,t)=0, && t>0
\\
&w(x,0) = w_{0}(x), \;\;\; w_{t}(x,0) = w_{1}(x), && x\in(0,1)
\end{aligned}
\right.
\end{equation}

The eigenvalue problem associated with~(\ref{waveeq}) is given by
\begin{align}
&u_{xx}-(\lambda^{2}+2\lambda a-b)u=0,\quad
x\in(0,1),\label{waveeig1}
\\
&u(0) = u(1)=0,\label{waveeig2}
\end{align}
and has received quite a lot of attention in the literature since
the papers of Chen et al.~\cite{chen} and Cox and Zuazua~\cite{cox}.
In the first of these papers the authors derived formally an
expression for the asymptotic behaviour of the eigenvalues
of~(\ref{waveeig1}),~(\ref{waveeig2}) in the case of a zero
potential $b$, which was later proved rigorously in the second of
the above papers. Following this, there were several papers on the
subject which, among other things, extended the results to
non--vanishing $b$~\cite{bera}, and showed that it is possible to
design damping terms which make the spectral abscissa as large as
desired~\cite{caco}. In~\cite{frei2} the second author of the
present paper addressed the inverse problem in arbitrary dimension
giving necessary conditions for a sequence to be the spectrum of an
operator of this type in the weakly damped case. As far as we are
aware, these are the only results for the inverse problem associated
with~(\ref{waveeig1}),~(\ref{waveeig2}). Other results for the
$n-$dimensional problem include, for instance, the fact that in that
case the decay rate is no longer determined solely by the
spectrum~\cite{lebe}, a study of some particular cases where the
role of geometric optics is considered~\cite{asle}, the asymptotic
behaviour of the spectrum~\cite{sjos} and the study of
sign--changing damping terms~\cite{frei1}.

The purpose of the present paper is twofold. On the one hand, we
show that problem~(\ref{waveeig1}),~(\ref{waveeig2}) may be
addressed in the same way as the classical Sturm--Liouville
problem in the sense that, although this is not a self--adjoinf
problem, the methods used for the former problem may be applied
here with similar results. This idea was already present in
both~\cite{chen} and~\cite{cox}. Here we take further advantage
of this fact to obtain the full asymptotic expansion for the
eigenvalues of~(\ref{waveeig1}),~(\ref{waveeig2})
(Theorem~\ref{thmasympt}). Based on these similarities, we were
also led to a  (regularized) trace formula in the spirit of that
for the Sturm--Liouville problem (Theorem~\ref{th1.4}).

On the other hand, the idea behind obtaining further terms in the
asyptotics was to use this information to address the associated
inverse spectral problem of finding all damping terms that give a
certain spectrum. Our main result along these lines is to show that
in the case of constant damping there is no other smooth damping
term yielding the same spectrum (Corollary~\ref{th1.2}). Namely, we
obtain the criterion for the damping term to be constant. Note that
this is in contrast with the inverse (Dirichlet) Sturm--Liouville
problem, where for each admissible spectrum there will exist a
continuum of potentials giving the same spectrum~\cite{potr}. In
particular this result shows that we should expect the inverse
problem to be much more rigid in  the case of the wave equation than
it is for the Sturm--Liouville problem. This should be understood in
the sense that, at least in the case of constant damping, it will
not be possible to perturb the damping term without disturbing the
spectrum, as is the case for the potential in the Sturm--Liouville
problem.

The plan of the paper is as follows. In the next section we set
the notation and state the main results of the paper. The proof
of the asymptotics of the eigenvalues is done in
Sections~\ref{sec3} and~\ref{sec4}, where in the first of these
we derive the form of the fundamental solutions of
equation~(\ref{waveeig1}), while in the second we apply a
shooting method to these solutions to obtain the formula for the
eigenvalues as zeros of an entire function -- the idea is the
same as that used in~\cite{cox}. Finally, in Section~\ref{sec5}
we prove the trace formula.

\section{Notation and results}

It is easy to check that if $\l$ is an eigenvalue of the problem
(\ref{waveeig1}), (\ref{waveeig2}), then $\overline{\l}$ is also
an eigenvalue of the same problem. In view of this property, we
denote the eigenvalues of this problem by $\l_n$, $n\not=0$, and
order them as follows
\begin{equation*}
\ldots\leqslant \IM\l_{-2}\leqslant \IM\l_{-1}\leqslant
\IM\l_1\leqslant \IM\l_2\leqslant \ldots
\end{equation*}
while assuming that $\l_{-n}=\overline{\l}_n$. We also suppose that
possible zero eigenvalues are $\l_{\pm 1}=\l_{\pm 2}=\ldots=\l_{\pm
p}=0$. If $p=0$, the problem (\ref{waveeig1}), (\ref{waveeig2}) has
no zero eigenvalues. For any function $f=f(x)$ we denote $\la
f\ra:=\int_0^1 f(x)\di x$. 
\begin{theorem}\label{thmasympt}
Suppose $a\in C^{m+1}[0,1]$, $b\in C^m[0,1]$, $m\geqslant 1$. The
eigenvalues of~(\ref{waveeig1}), (\ref{waveeig2}) have the following
asymptotic behaviour as $n\to\pm\infty$:
\begin{align}
\l_n&=\pi n \iu +\sum\limits_{j=0}^{m-1} c_j n^{-j}+\Odr(n^{-m}),
\label{1.4}
\end{align}
were the $c_j$'s are numbers which can be determined explicitly.
In particular,
\begin{align}
& c_0 =-\la a\ra, \quad c_1=\frac{\la
a^2+b\ra}{2\pi\iu},\label{1.5}
\\
& c_2=\frac{1}{2\pi^2}\left[\la a(a^2+b)\ra-\la a\ra\la a^2+b\ra
+\frac{a'(1)-a'(0)}{2}\right]. \label{1.6}
\end{align}
\end{theorem}


A straightforward consequence of the fact that the spectrum determines
the average as well as the $L^{2}$ norm of the damping term (assuming
$b$ fixed) is that the spectrum corresponding to the constant damping
determines this damping uniquely.

\begin{corollary}\label{th1.2}
Assume that $a\in C^3[0,1]$, $\l_n$ are the eigenvalues of the
problem (\ref{waveeig1}), (\ref{waveeig2}), the function $b\in
C^2[0,1]$ is fixed, and the formula (\ref{1.4}) gives the
asymptotics for these eigenvalues. Then the function $a(x)$ is
constant, if and only if
\begin{equation*}
c_0^2=2\pi\iu c_1-\la b\ra,
\end{equation*}
in which case $a(x)\equiv -c_0$.
%
\end{corollary}

In the same way, the asymptotic expansion allows us to
derive other spectral invariants in terms of the damping term $a$.
However, these do not have such a simple interpretation as in the
case of the above constant damping result.
\begin{corollary}\label{th1.3}
Suppose $b\equiv0$, $a_i(x)=a_0(x)+\widetilde{a}_i(x)$, $i=1,2$,
where $a_0(1-x)=a_0(x)$, $\widetilde{a}_i(1-x)=-\widetilde{a}_i(x)$,
$\widetilde{a}_i, a_0\in C^4[0,1]$, and for $a=a_i$ the problems
(\ref{waveeig1}), (\ref{waveeig2}) have the same spectra. Then
\begin{equation*}
\la \widetilde{a}_1^2\ra=\la \widetilde{a}_2^2\ra,\quad \la
a_0\widetilde{a}_1^2\ra=\la a_0\widetilde{a}_2^2\ra
\end{equation*}
is valid.
\end{corollary}

From Theorem~\ref{thmasympt} we have that the quantity
$\RE(\l_n-c_0)$ behaves as $\Odr(n^{-2})$ as $n\to\infty$. This
means that the series
\begin{equation*}
\sum\limits_{\genfrac{}{}{0pt}{}{n=-\infty}{n\not=0}}^{\infty}
(\l_n-c_0)= 2\sum\limits_{n=1}^{\infty} \RE(\l_n-c_0)
\end{equation*}
converges. In the following theorem we express the sum of this
series in terms of the function $a$. This is in fact the
formula for the regularized trace.
\begin{theorem}\label{th1.4}
Let $a\in C^3[0,1]$, $b\in C^2[0,1]$. Then the identity
\begin{equation*}
\sum\limits_{\genfrac{}{}{0pt}{}{n=-\infty}{n\not=0}}^{\infty}
(\l_n-c_0)=\frac{a(0)+a(1)}{2}-\la a\ra
\end{equation*}
holds.
\end{theorem}

\section{Asymptotics for the fundamental system\label{sec3}}

In this section we obtain the asymptotic expansion for the
fundamental system of the solutions of the equation (\ref{waveeig1})
as $\l\to\infty$, $\l\in \mathbb{C}$. This is done by means of the
standard technique described in, for instance, \cite[Ch. I\!V, Sec.
4.2, 4.3]{Er}, \cite[Ch. I\!I, Sec. 3]{F}.

We begin with the formal construction assuming the
asymptotics to be of the form
\begin{equation}\label{2.1}
u_\pm(x,\l)=\E^{\pm \l x\pm \int\limits_0^x\phi_\pm(t,\l)\di t},
\end{equation}
where
\begin{equation}\label{2.1a}
\phi_\pm(x,\l)=\sum\limits_{i=0}^{m}
\phi_i^{(\pm)}(x)\l^{-i}+\Odr(\l^{-m-1}),\quad m\geqslant 1.
\end{equation}
In what follows we assume that $a\in C^{m+1}[0,1]$, $b\in C^m[0,1]$.

We substitute the series (\ref{2.1}), (\ref{2.1a}) into
(\ref{waveeig1}) and equate the coefficients of the same powers of
$\l$. It leads us to a recurrent system of equations determining
$\phi_i^{(\pm)}$ which read as follows:
\begin{align}
&\phi_0^{(\pm)}=a, \label{2.2}
\\
&\phi_1^{(\pm)}=-\frac{1}{2}(\pm a'+a^2+b),\label{2.3}
\\
&\phi_i^{(\pm)}=-\frac{1}{2}\left(\pm{\phi_{i-1}^{(\pm)}}'
+\sum\limits_{j=0}^{i-1}\phi_j^{(\pm)}\phi_{i-j-1}^{(\pm)}
\right),\quad i\geqslant 2.
\end{align}
The main aim of this section is to prove that there exist solutions
to (\ref{waveeig1}) having the asymptotics (\ref{2.1}),
(\ref{2.1a}). In other words, we are going to justify these
asymptotics rigorously. We will do this for $u_+$, the case of $u_-$
following along similar lines.

Let us write
\begin{equation*}
U_m(x,\l)=\E^{\l x+\sum\limits_{i=0}^{m}\l^{-i}\int\limits_{0}^x
\phi_i^{(+)}(t)\di t}.
\end{equation*}
In view of the assumed smoothness for $a$ and $b$ we conclude that
$U_m\in C^2[0,1]$. It  is also easy to check that
\begin{equation}\label{2.7}
\begin{aligned}
&U_m''-\l^2 U_m-2\l a U_m+b U_m=\l^{-m}\E^{\l x}f_m(x,\l),\quad
x\in[0,1],
\\
&U_m(0)=1,\quad
U'_m(0)=\l+\sum\limits_{i=0}^{m}\phi_i^{(+)}(0)\l^{-i},
\end{aligned}
\end{equation}
where the function $f_m$ satisfies the estimate
\begin{equation*}
|f_m(x,\l)|\leqslant C_m
\end{equation*}
uniformly for large $\l$ and $x\in[0,1]$

We consider first the case $\RE\l\geqslant 0$. Differentiating
the function $u_+$ formally we see that
\begin{equation*}
u'_+(0,\l)=\l+\phi_+(0,\l)=\l+\sum\limits_{i=0}^{m} \phi_i^{(+)}
(0)\l^{-i}+\Odr(\l^{-m-1}).
\end{equation*}
Let 
\begin{equation*}
A_0(\l)=\l+\sum\limits_{i=0}^{m} \phi_i^{(+)}(0)\l^{-i}, 
\end{equation*}
\NB and $u_+(x,\l)$ be the solution to the Cauchy problem for the
equation (\ref{waveeig1}) subject to the initial conditions
\begin{equation*}
u_+(0,\l)=1,\quad u_+'(0,\l)=
A_0(\l).
\end{equation*}
We introduce one more function $w_m(x,\l)=u_+(x,\l)/U_m(x,\l)$. This
function solves the Cauchy problem
\begin{align*}
&(U_m^2 w_m')'+\l^{-m}U_m\E^{\l x}f_m w_m=0,\quad x\in[0,1],
\\
& w_m(0,\l)=1,\quad w'_m(0,\l)=0. 
\end{align*}
The last problem is equivalent to the integral equation
\begin{align*}
&w_m(x,\l)+\l^{-m}(K_m(\l) w_m)(x,\l)=
1,
\\
&(K_m(\l)w_m)(x,\l):=\int\limits_0^x
U_m^{-2}(t_1)\int\limits_0^{t_1} U_m(t_2) \E^{\l t_2} f_m(t_2,\l)
w_m(t_2,\l)\di t_2 \di t_1.
\end{align*}
Since $\RE\l\geqslant 0$ for $0\geqslant t_2\geqslant
t_1\geqslant 1$, the estimate
\begin{equation*}
|U_m^{-2}(t_1,\l)U_m(t_2,\l)\E^{\l t_1}|\leqslant C_m
\end{equation*}
holds true, where the constant $C_m$ is independent of $\l$, $t_1$,
$t_2$. Hence, the integral operator $K_m: C[0,1]\to C[0,1]$ is
bounded uniformly in $\l$ large enough, $\RE \l\geqslant 0$.
Employing this fact, we conclude that
\begin{equation*}
w_m(x)=1+\Odr(\l^{-m}),\qquad \l\to\infty,\quad\RE\l\geqslant0,
\end{equation*}
in the $C^2[0,1]$-norm. Hence, the formula (\ref{2.1}), where
\begin{equation}\label{3.13}
\phi_+(x,\l)=\sum\limits_{i=0}^{m-1}
\phi_i^{(+)}(x)\l^{-i}+\Odr(\l^{-m}),
\end{equation}
gives the asymptotic expansion for the solution of the Cauchy
problem (\ref{waveeig1}), (\ref{2.7}) as $\l\to\infty$,
$\RE\l\geqslant0$.

Suppose now that $\RE\l\leqslant0$. Let $A_1(\l)$, $A_2(\l)$ be
functions having the asymptotic expansions
\begin{equation*}
A_1(\l)=\l+\sum\limits_{i=0}^m
\l^{-i}\int\limits_0^1\phi_i^{(+)}(x)\di x,\quad
A_2(\l)=\l+\sum\limits_{i=0}^m \phi_i^{(+)}(1)\l^{-i}.
\end{equation*}
We define the function $\widetilde{u}_+(x,\l)$ as the solution
to the Cauchy problem for equation (\ref{waveeig1}) subject to
the initial conditions
\begin{equation*}
\widetilde{u}_+(1,\l)=\E^{A_1(\l)},\quad
\widetilde{u}'_+(1,\l)=A_2(\l)\E^{A_1(\l)}.
\end{equation*}
In a way analogous to the arguments given above, it is possible to check
that the function $\widetilde{u}_+$ has the asymptotic expansion
(\ref{2.1}) in the $C^2[0,1]$-norm as $\l\to+\infty$, $\RE\l\leqslant
0$. Hence, $\widetilde{u}_+(0,\l)=1+\Odr(\l^{-m})$ for each
$m\geqslant 1$. In view of this identity we conclude that the
function $u_+(x,\l):=\widetilde{u}_+(x,\l)/\widetilde{u}_+(0,\l)$ is
a solution to (\ref{waveeig1}), satisfies the condition
$u_+(0,\l)=1$, and has the asymptotic expansion (\ref{2.1}), where
the asymptotics for  $\phi_+$ is given in (\ref{3.13}).

For convenience we summarize the obtained results in
\begin{lemma}\label{lm2.1}
Let $a\in C^{m+1}[0,1]$, $b\in C^m[0,1]$. There exist two linear
independent solutions to the equation (\ref{waveeig1}) satisfying
the initial condition $u_\pm(0,\l)=1$ and having the asymptotic
expansions (\ref{2.1}) in the $C^2[0,1]$-norm as $\l\to\infty$,
$\l\in\mathbb{C}$, where
\begin{equation*}
\phi_\pm(x,\l)=\sum\limits_{i=0}^{m-1}
\phi_i^{(\pm)}(x)\l^{-i}+\Odr(\l^{-m}).
\end{equation*}
\end{lemma}

\section{Asymptotics of the eigenvalues\label{sec4}}

This section is devoted to the proof of Theorem~\ref{thmasympt} and
Corollaries~\ref{th1.2} and~\ref{th1.3}. We assume that $a\in
C^{m+1}[0,1]$, $b\in C^m[0,1]$, $m\geqslant 1$.

Let $u=u(x,\l)$ be the solution to (\ref{waveeig1}) subject to
the initial conditions $u(0,\l)=0$, $u'(0,\l)=1$. Denote
$\g_0(\l):=u(1,\l)$. The function $\g_0$ is entire, and its zeros
coincide with the eigenvalues of the problem
(\ref{waveeig1}), (\ref{waveeig2}). It follows from
Lemma~\ref{lm2.1} that, for $\l$ large enough the function $u(x,\l)$
can be expressed in terms of $u_\pm$ by
\begin{equation*}
u(x,\l)=\frac{u_+(x,\l)-u_-(x,\l)}{u'_+(0,\l)-u'_-(0,\l)}.
\end{equation*}
The denominator is non-zero, since due to (\ref{2.1})
\begin{equation*}
u'_+(0,\l)-u'_-(0,\l)=2\l+2\la a\ra+\Odr(\l^{-1}),\quad
\l\to\infty.
\end{equation*}
Thus, for $\l$ large enough
\begin{equation}\label{3.3}
\g_0(\l)=\frac{u_+(1,\l)-u_-(1,\l)}{u'_+(0,\l)-u'_-(0,\l)}.
\end{equation}

\begin{lemma}\label{lm3.1}
For $n$ large enough, the set
\[
Q:=\{\l: |\RE\l|<\pi n+\pi/2, |\IM\l|<\pi n+\pi/2\}
\] contains exactly $2n$ eigenvalues of
the problem (\ref{2.1}), (\ref{2.2}).
\end{lemma}

\begin{proof}
Let $\g_1(\l):=\g_0(\l)\E^{\l+\la a\ra}$. The zeros of $\g_1$ are
those of $\g_0(\l)$. For $\l$ large enough we represent the function
$\g_1(\l)$ as \NB
\begin{align*}
\g_1(\l)=&\g_2(\l)+\g_3(\l),\quad \g_2:=\frac{\E^{2(\l+
a(0))}-1}{2(\l+a(0))},
\\
\g_3(\l)=&-\g_2(\l)\frac{\widetilde{\phi}_+(0,\l)+
\widetilde{\phi}_-(0,\l)+2(1+\l^{-1}a(0))(1-\E^{\l^{-1}\la
\widetilde{\phi}_+(\cdot,\l)\ra}) }{2\l(\l+a(0))+
\widetilde{\phi}_+(0,\l)+ \widetilde{\phi}_-(0,\l)}
\\
&+ \frac{\E^{\l^{-1}\la\widetilde{\phi}_+(\cdot,\l)\ra} -
\E^{-\l^{-1}\la\widetilde{\phi}_-(\cdot,\l)\ra}}{2(\l+a(0))+\l^{-1}
(\widetilde{\phi}_+(0,\l)+ \widetilde{\phi}_-(0,\l))},
\\
\widetilde{\phi}_\pm(x,\l):=&\l^{-1}(\phi_\pm(x,\l)-a(x)).
\end{align*}
It is clear that for $\l$ large enough  the function $\g_3(\l)$
satisfies an uniform in $\l$ estimate
\begin{equation*}
|\g_3(\l)|\leqslant C|\l|^{-2} \big(|\g_2(\l)|+1\big).
\end{equation*}
One can also check easily that \NB
\begin{equation*}
|\g_2(\l)|\geqslant C|\l|,\quad \l\in \p K,
\end{equation*}
if $n$ is large enough. These two last estimates imply that
$|\g_3(\l)|\leqslant |\g_2(\l)|$ as $\l\in\p K$, if $n$ is large
enough. By Rouch\'e theorem we conclude that for such $n$ the
function $\g_1$ has the same amount of zeros inside $Q$ as the
function $\g_2$ does. Since the zeros of the latter are given by
$\pi n \iu-\la a\ra$, $n\not=0$, this completes the proof.
\end{proof}

\begin{proof}[Proof of Theorem~\ref{thmasympt}]
Assume first that $a\in C^2[0,1]$, $b\in C^1[0,1]$. As was mentioned
above, the eigenvalues of problem (\ref{waveeig1}), (\ref{waveeig2})
are the zeros of the function $\g_0(\l)=0$. It follows from
Lemma~\ref{lm3.1} that these eigenvalues tend to infinity as
$n\to\infty$. By Lemma~\ref{lm2.1}, for $\l$ large enough the
equation $\g_0(\l)=0$ becomes
\begin{equation*}
\E^{2\l+\la \phi_+(\cdot,\l)+\phi_-(\cdot,\l)\ra}=0
\end{equation*}
which may be rewritten as
\begin{equation}\label{3.4}
2\l+\la \phi_+(\cdot,\l)+\phi_-(\cdot,\l)\ra=2\pi n\iu,\quad
n\in \mathbb{Z}.
\end{equation}
If we now replace $\phi_\pm$ by the leading terms of their
asymptotic expansions we obtain
\begin{align}
&2\l+2\la a\ra+\Odr(\l^{-1})=2\pi n\iu,\label{3.4a}
\\
&\l=\pi n \iu-\la a\ra+o(1),\quad n\to\infty.\nonumber
\end{align}
Hence, the eigenvalues behave as $\l\sim \pi n \iu -\la a\ra$
for large $n$. Moreover, it follows from Lemma~\ref{lm3.1} that
it is exactly the eigenvalue $\l_n$ which behaves as
\begin{equation*}
\l_n=\pi n \iu-\la  a\ra+o(1),\quad n\to\infty.
\end{equation*}
It follows from this identity and (\ref{3.4a}) that
\begin{equation*}
\l_n=\pi n \iu-\la  a\ra+\Odr(n^{-1}),\quad n\to\infty,
\end{equation*}
and we complete the proof in the case $m=1$. If $m=2$, we substitute
the above identity and (\ref{2.1}) into (\ref{3.4}) and get
\begin{align*}
&\l_n+\la a\ra+\frac{1}{\l_n}\la \phi_1^{(+)}+\phi_1^{(-)}\ra
+\Odr(\l_n^{-2})=\pi n \iu,
\\
&\l_n=\pi n \iu-\la a\ra-\frac{\la \phi_1^{(+)}+\phi_1^{(-)}\ra}{\pi
n \iu}+\Odr(n^{-2}).
\end{align*}
The last formula and the identities (\ref{2.3}) yield
formulas (\ref{1.5}) for $c_0$ and $c_1$. Repeating the described
procedure one can easily check that the asymptotics (\ref{1.4}),
(\ref{1.5}) hold true.
\end{proof}

\begin{proof}[Proof of Corollary~\ref{th1.2}]
The coefficients $c_0$, $c_1$ in the asymptotics (\ref{1.4}) are
determined by the formulas (\ref{1.5}) and, by the Cauchy-Schwarz
inequality, we thus obtain
\begin{equation*}
c_0^2=\la a\ra^2\leqslant \la a^2\ra=2\pi \iu c_1-\la b\ra,
\end{equation*}
with equality if and only if
$a(x)$ is a constant function. This fact completes the proof.
\end{proof}

\begin{proof}[Proof of Corollary~\ref{th1.3}]
It follows from (\ref{1.5}), (\ref{1.6}) that
$\la a_1^2\ra=\la a_2^2\ra$, $\la a_1^3\ra=\la a_2^3\ra$.
Now we check that
\begin{equation*}
\la a_i^2\ra=\la a_0^2\ra+\la \widetilde{a}_i^2\ra,\quad \la
a_i^3\ra=\la a_0^3\ra+3\la a_0 \widetilde{a}_i^2\ra,\quad i=1,2,
\end{equation*}
and arrive at the statement of the theorem.
\end{proof}

\section{Regularized trace formulas\label{sec5}}

In this section we prove Theorem~\ref{th1.4}. We follow the idea
employed in the proof of the similar trace formula for the
Sturm-Liouville operators in \cite[Ch. I, Sec. 14]{LS}.

We begin by defining the function
\begin{equation*}
\Phi(\l):=\l^{2p} \prod\limits_{n=p+1}^\infty
\left(1-\frac{\l}{\l_n}\right)
\left(1-\frac{\l}{\overline{\l}_n}\right).
\end{equation*}
The above product converges, since
\begin{equation*}
\left(1-\frac{\l}{\l_n}\right)
\left(1-\frac{\l}{\overline{\l}_n}\right)=1+\frac{\l^2-2\l\RE\l_n}
{|\l_n|^2},
\end{equation*}
and by Theorem~\ref{thmasympt} we have
\begin{equation}\label{4.0}
\begin{aligned}
&|\l_n|^2=\pi^2 n^2-2\pi\iu c_1+c_0^2+\Odr(n^{-2}),
\\
&\RE\l_n=c_0+\Odr(n^{-2})
\end{aligned}
\end{equation}
as $n\to+\infty$. Proceeding in the same way as in the formulas
(14.8), (14.9) in \cite[Ch. I, Sec. 14]{LS}, we obtain
\begin{align*}
&\Phi(\l)=\frac{C_0\Psi(\l)\sinh\l}{\l},
\\
&\Psi(\l):=\prod\limits_{n=1}^\infty \left(1-\frac{\pi^2
n^2-|\l_n|^2+2\l\RE\l_n}{\pi^2 n^2+\l^2}\right),
\\
&C_0:=(\pi n)^{2p}\prod\limits_{n=p+1}^\infty \frac{\pi^2
n^2}{|\l_n|^2}.
\end{align*}
In what follows we assume that $\l$ is real, positive and
large. In the same way as in \cite[Ch. I, Sec. 14]{LS} it is possible to
derive the formula
\begin{equation}\label{4.2}
\ln\Psi(\l)=-\sum\limits_{k=1}^{\infty}
\frac{1}{k}\sum\limits_{n=1}^{\infty} \left(\frac{\pi^2
n^2-|\l_n|^2+2\l\RE\l_n}{\pi^2 n^2+\l^2}\right)^k.
\end{equation}
Our aim is to study the asymptotic behaviour of $\ln\Psi(\l)$ as
$\l\to+\infty$. Employing the same arguments as in the proof of
Lemma~14.1 and in the equation (14.11) in
\cite[Ch. I, Sec. 14]{LS}, we arrive at the estimate
\begin{align}
& \sum\limits_{n=1}^{\infty} \left( \frac{\pi^2 n^2-|\l_n|^2+2\l
\RE\l_n}{\pi^2 n^2+\l^2}\right)^k\leqslant c^k\l^k
\sum\limits_{n=1}^{\infty} \frac{1}{\left(\pi^2 n^2+\l^2\right)^k}
\nonumber
\\
&\hphantom{\sum\limits_{n=1}^{\infty}} \leqslant c^k \l^k
\int\limits_{0}^{+\infty} \frac{\di t}{(\pi^2
t^2+\l^2)^k}=\frac{c^k}{\l^k}\int\limits_{0}^{+\infty}\frac{\di
z}{(\pi^2 z^2+1)^k}\leqslant \frac{c^{k+1}}{\l^k},\nonumber
\\
&\sum\limits_{k=3}^{\infty} \frac{1}{k}\sum\limits_{n=1}^{\infty}
\left(\frac{\pi^2 n^2-|\l_n|^2+2\l\RE\l_n}{\pi^2
n^2+\l^2}\right)^k=\Odr(\l^{-3}), \quad \l\to+\infty,\label{4.3}
\end{align}
where $c$ is a constant independent of $k$ and $n$. Let us analyze
the asymptotic behaviour of the first two terms in the series
(\ref{4.2}). As $k=1$, we have
\begin{equation}\label{4.4}
\begin{aligned}
&\sum\limits_{n=1}^{\infty} \frac{\pi^2
n^2-|\l_n|^2+2\l\RE\l_n}{\pi^2 n^2+\l^2}=
\sum\limits_{n=1}^{\infty} \frac{\pi^2 n^2-|\l_n|^2-2\pi\iu
c_1+c_0^2}{\pi^2 n^2+\l^2}
\\
&\hphantom{\sum\limits_{n=1}^{\infty}\frac{\pi^2
n^2-|\l_n|^2+2\l}{\pi^2 n^2+\l^2}}+\left(2\pi\iu c_1-c_0^2+2\l
c_0\right)\sum\limits_{n=1}^{\infty}\frac{1}{\pi^2 n^2+\l^2}
\\
&\hphantom{\sum\limits_{n=1}^{\infty}\frac{\pi^2
n^2-|\l_n|^2+2\l}{\pi^2 n^2+\l^2}}+ 2\l^{-1} S -
2\l^{-1}\sum\limits_{n=1}^{\infty} \frac{\pi^2
n^2(\RE\l_n-c_0)}{\pi^2 n^2+\l^2},
\\
&\mbox{where } S:=\sum\limits_{n=1}^{\infty} (\RE\l_n-c_0).
\end{aligned}
\end{equation}
Taking into account (\ref{4.0}), we have 
\begin{align*}
&\left|\sum\limits_{n=1}^{\infty} \frac{\pi^2
n^2-|\l_n|^2-2\pi\iu c_1+c_0^2}{\pi^2 n^2+\l^2}\right|\leqslant
C\sum\limits_{n=1}^{\infty} \frac{1}{n^2(\pi^2 n^2+\l^2)}
\\
&\hphantom{\sum\limits_{n=1}^{\infty}\frac{\pi^2
n^2-|\l_n|^2}{\pi^2 n^2+\l^2}}=
\frac{\ds \pi^2}{\ds 6}\frac{\ds 3+\l^2-3\coth\l}{\ds \l^4}\leqslant C\l^{-2},
\\
&\left|\sum\limits_{n=1}^{\infty} \frac{(\RE\l_n-c_0)\pi^2 n^2
}{\pi^2n^2+\l^2}\right|\leqslant
C\sum\limits_{n=1}^{\infty}\frac{\ds 1}{\ds \pi^2 n^2+\l^2}\leqslant C\l^{-1},
\end{align*}
where the constant $C$ is independent of $\l$. Here we have also
used the formula 
\begin{equation}\label{4.5a}
\sum\limits_{n=1}^{\infty} \frac{1}{\pi^2
n^2+\l^2}=\frac{\l\coth\l-1}{2\l^2}=\frac{\l^{-1}-\l^{-2}}{2}+
\Odr(\l^{-1}\E^{-2\l})
\end{equation}
as $\l\to+\infty$. We employ this formula to calculate the
remaining terms in (\ref{4.4}) and arrive at the identity
\begin{equation}\label{4.5}
\sum\limits_{n=1}^{\infty}\frac{\pi^2
n^2-|\l_n|^2+2\l\RE\l_n}{\pi^2n^2+\l^2}=c_0+\left(2S-c_0-
\frac{c_0^2}{2}+\iu\pi c_1\right)\l^{-1}+\Odr(\l^{-2}),
\end{equation}
as $\l\to+\infty$. For $k=2$ we proceed in the similar way,
\begin{align*}
&\sum\limits_{n=1}^{\infty} \left(\frac{\pi^2
n^2-|\l_n|^2+2\l\RE\l_n}{\pi^2 n^2+\l^2}\right)^2=
\sum\limits_{n=1}^{\infty} \frac{(\pi^2
n^2-|\l_n|^2)^2}{(\pi^2n^2+\l^2)^2}
\\
&\hphantom{\sum\limits_{1}^{1}\pi^2 n^2}-
2\l\sum\limits_{n=1}^{\infty}\frac{(\pi^2
n^2-|\l_n|^2)\RE\l_n}{(\pi^2 n^2+\l^2)^2}+ 4\l^2 c_0^2
\sum\limits_{n=1}^{\infty} \frac{1}{(\pi^2 n^2+\l^2)^2}
\\
&\hphantom{\sum\limits_{1}^{1}\pi^2 n^2}+4\l^2
\sum\limits_{n=1}^{\infty} \frac{(\RE\l_n)^2-c_0^2}{(\pi^2
n^2+\l^2)^2}.
\end{align*}
By differentiating (\ref{4.5a}) we obtain
\begin{equation*}
\sum\limits_{n=1}^{\infty}\frac{1}{(\pi^2 n^2+\l^2)^2}=
\frac{\l\coth\l-2-\l^2(1-\coth^2\l)}{4\l^4}.
\end{equation*}
This identity and (\ref{4.0}) yield that as $\l\to+\infty$
\begin{align*}
&\left|\sum\limits_{n=1}^{\infty} \frac{(\pi^2
n^2-|\l_n|^2)^2}{(\pi^2n^2+\l^2)^2}\right|\leqslant
C\sum\limits_{n=1}^{\infty} \frac{1}{(\pi^2n^2+\l^2)^2}\leqslant
C\l^{-3},
\\
&\left|\sum\limits_{n=1}^{\infty} \frac{(\pi^2
n^2-|\l_n|^2)\RE\l_n}{(\pi^2n^2+\l^2)^2}\right|\leqslant
C\sum\limits_{n=1}^{\infty} \frac{1}{(\pi^2n^2+\l^2)^2}\leqslant
C\l^{-3},
\\
&\left|\sum\limits_{n=1}^{\infty}
\frac{(\RE\l_n)^2-c_0^2}{(\pi^2n^2+\l^2)^2}\right|\leqslant
C\sum\limits_{n=1}^{\infty} \frac{1}{(\pi^2n^2+\l^2)^2}\leqslant
C\l^{-3}.
\end{align*}
Hence,
\begin{equation}\label{4.6}
\sum\limits_{n=1}^{\infty} \left(\frac{\pi^2
n^2-|\l_n|^2+2\l\RE\l_n}{\pi^2
n^2+\l^2}\right)^2=c_0^2\l^{-1}+\Odr(\l^{-2})
\end{equation}
as $\l\to+\infty$. It follows from (\ref{4.2}), (\ref{4.3}),
(\ref{4.5}), (\ref{4.6}) that
\begin{align}
&\ln\Psi(\l)=-c_0-(2S-c_0+\iu\pi
c_1)\l^{-1}+\Odr(\l^{-2}),\nonumber
\\
&\Phi(\l)=\frac{C_0\E^{-c_0}\sinh\l}{\l}\big[1-(2S-c_0+\iu\pi
c_1)\l^{-1}+\Odr(\l^{-2})\big]\label{4.7}
\end{align}
as $\l\to+\infty$. It follows from (\ref{3.3}) and
Lemma~\ref{lm2.1} that for $\l$ large enough the estimate 
\NB
\begin{equation*}
|\g_0(\l)|\leqslant C|\l|^{-1}\E^{|\l|}
\end{equation*}
holds true. Hence, the order of the entire function $\g_0(\l)$ is
one. In view of Theorem~\ref{thmasympt} we also conclude that the
series $\sum\limits_{n=p+1}^{\infty} |\l_n|^{-2}$ converges and
therefore the genus of the canonical product associated with $\g_0$
is one. We apply Hadamard's theorem (see, for instance, \cite[Ch. I,
Sec. 10, Th. 13]{L}) and obtain that
\begin{equation*}
\g_0(\l)=\E^{P(\l)}
\Phi(\l), \quad P(\l)=\a_1\l+\a_0+2\sum\limits_{n=p+1}^{\infty}
|\l_n|^{-2}\RE\l_n,
\end{equation*}
where $\a_1$, $\a_0$ are some numbers. Hence, due to
(\ref{4.7}), it follows that $\g_0$ behaves as
\begin{align*}
&\g_0(\l)=\frac{C_0\E^{P(\l)}\sinh\l}{\l} \big[ 1-(2S-c_0+\iu\pi
c_1)\l^{-1}+\Odr(\l^{-2})\big],
\end{align*}
as $\l\to+\infty$.
On the other hand, Lemma~\ref{lm2.1} and (\ref{3.3}) imply
that
\begin{equation*}
\g_0(\l)=\frac{\E^{\l+\la a\ra}}{2\l}\big[1+(\la
\phi_1^{(+)}\ra-a(0))\l^{-1}+\Odr(\l^{-2})\big]+\Odr(\l^{-1}\E^{-\l}),
\end{equation*}
as $\l\to+\infty$. Comparing the last two identities yields
$\a_1=0$,
\begin{equation*}
C_0\E^{\a_0-c_0+2\sum\limits_{n=1}^{\infty}
|\l_n|^{-2}\RE\l_n}=\E^{\la a\ra}
\end{equation*}
and
\begin{equation*}
-(2S-c_0+\iu\pi
c_1)=\la\phi_1^{(+)}\ra-a(0).
\end{equation*}
It now follows from (\ref{1.5}), (\ref{2.3}) that
\begin{equation*}
\sum\limits_{\genfrac{}{}{0pt}{}{n=-\infty}{n\not=0}}^{\infty}
(\l_n-c_0)= 2S=c_0+a(0)-\la\phi_1^{(+)}\ra-\iu\pi
c_1=\frac{a(0)+a(1)}{2}-\la a\ra,
\end{equation*}
completing the proof of Theorem~\ref{th1.4}.

\section*{Acknowledgments}
This work was done during the visit of
D.B. to the Universidade de Lisboa; he is grateful for the
hospitality extended to him. P.F. would like
to thank A. Laptev for several conversations of this topic.

\end{document}